\documentclass[12pt]{amsart}
\usepackage[utf8]{inputenc}
\usepackage{amsmath,amssymb}
\usepackage{multirow}
\usepackage{hyperref}

\title{Composite Rational Functions and Arithmetic Progressions}
 
\author{Sz. Tengely}
\address{
Mathematical Institute\newline
\indent University of Derecen\newline
\indent P.O. Box 12\newline
\indent 4010 Debrecen\newline
\indent Hungary}
\email{tengely@science.unideb.hu}
\thanks{Research supported in part by the OTKA grants NK104208 and K100339}
\newtheorem{thm}{Theorem}
\newtheorem*{thmA}{Theorem A}

\newtheorem{prop}{Proposition}

\theoremstyle{definition}
\newtheorem*{rem}{Remark}

\begin{document}
\newcommand{\Q}{{\mathbb Q}}
\newcommand{\Z}{{\mathbb Z}}
\newcommand{\N}{{\mathbb N}}
\newcommand{\C}{{\mathbb C}}
\newcommand{\cK}{\mathcal{K}}
\newcommand{\OO}{{\mathcal O}}
\newcommand{\ord}{\mbox{\rm ord}}
\newcommand{\lcm}{\mbox{\rm lcm}}
\newcommand{\sign}{\mbox{\rm sign}}

\begin{abstract}
In this paper we deal with composite rational functions having zeros and poles forming consecutive elements of an arithmetic progression.
We also correct a result published in \cite{PT-CRF} related to composite rational functions having a fixed 
number of zeros and poles.
\end{abstract}

\bibliographystyle{plain}
\maketitle

\section{Introduction}
We consider a problem related to decompositions of polynomials and rational functions. In this subject a classical result obtained by Ritt \cite{Ritt} says 
that if there is a polynomial $f\in\mathbb{C}[X]$ satisfying certain tameness properties
and 
$$
f=g_1\circ g_2\circ\cdots\circ g_r=h_1\circ h_2\circ\cdots\circ h_s,
$$
then $r=s$ and $\{\deg g_1,\ldots,\deg g_r\}=\{\deg h_1,\ldots,\deg h_r\}.$ Ritt's fundamental result has been investigated, extended and applied
in various wide-ranging contexts (see e.g. \cite{BeNg, Fried1, FriedDioph, GutSev2, GutSev1, LeviRitt, ZannierRitt, Zcompoly}).
The above mentioned result is not valid for rational functions. Gutierrez and Sevilla \cite{GutSev2}
provided the following example 
\begin{eqnarray*}
f=\frac{x^3(x+6)^3(x^2-6x+36)^3}{(x-3)^3(x^2+3x+9)^3},\\
f=g_1\circ g_2\circ g_3=x^3\circ \frac{x(x-12)}{x-3}\circ \frac{x(x+6)}{x-3},\\
f=h_1\circ h_2=\frac{x^3(x+24)}{x-3}\circ \frac{x(x^2-6x+36)}{x^2+3x+9}.
\end{eqnarray*}
To determine decompositions of a given rational function there were developed algorithms (see e.g. \cite{AlonsoAlg, AyadAlg, BartonZippel}). 
In \cite{AyadAlg}, Ayad and Fleischmann implemented a MAGMA \cite{MAGMA} code to find decompositions, they provided the following example
$$
f=\frac{x^4-8x}{x^3+1}
$$
and they obtained that $f(x)=g(h(x)),$ where
$$
g=\frac{x^2+4x}{x+1}\quad\mbox{and}\quad h=\frac{x^2-2x}{x+1}.
$$
Fuchs and Peth\H{o} \cite{FPcrf} proved the following theorem.
\begin{thmA}\label{A}
Let $k$ be an algebraically closed field of characteristic zero.
Let $n$ be a positive integer. Then there exists a positive
integer $J$ and, for every $i\in\{1,\ldots,J\}$, an affine algebraic variety $V_i$ defined
over $\Q$ and with $V_i \subset \mathbb{A}^{n+t_i}$ for some $2 \leq t_i \leq n$, such that:

(i) If $f, g, h \in k(x)$ with $f(x)=g(h(x))$ and with $\deg g, \deg h \geq 2, g$ not
of the shape $(\lambda(x))^m , m \in \N, \lambda \in PGL_2(k),$ and $f$ has at most $n$ zeros
and poles altogether, then there exists for some $i\in\{1,\ldots, J\}$ a point
$P = (\alpha_1,\ldots,\alpha_n,\beta_1,\ldots,\beta_{t_i}) \in V_i(k)$, a vector 
$(k_1,\ldots,k_{t_i})\in \Z^{t_i}$
with $k_1 + k_2 + \ldots + k_{t_i} = 0$ depending only \footnote{in \cite{FPcrf} it is written as "or not depending", this typo is corrected here.} on $V_i$ , a partition of
$\{1,\ldots,n\}$ in $t_i + 1$ disjoint sets $S_{\infty}, S_{\beta_1},\ldots, S_{\beta_{t_i}}$ with 
$S_{\infty} = \emptyset$ if $k_1 + k_2 +\ldots+ k_{t_i} = 0$, and a vector $(l_1,\ldots,l_n)\in\{0,1,\ldots,n-1\}^n$, also both depending only on $V_i$, such that
\begin{equation*}
f(x)=\prod_{j=1}^{t_i}(\omega_j/\omega_{\infty})^{k_j},\quad 
g(x)=\prod_{j=1}^{t_i}(x-\beta_j)^{k_j}
\end{equation*}
and
\begin{equation*}
h(x)=
\begin{cases}
\beta_j+\frac{\omega_j}{\omega_{\infty}}\quad (j=1,\ldots,t_i), & \mbox{ if } k_1+k_2+\ldots+k_{t_i}\neq 0\\
\frac{\beta_{j_1}\omega_{j_2}-\beta_{j_2}\omega_{j_1}}{\omega_{j_2}-\omega_{j_1}}\quad (1\leq j_1<j_2\leq t_i), & \mbox{ otherwise,}
\end{cases}
\end{equation*}
where
\begin{equation*}
\omega_j=\prod_{m\in S_{\beta_j}}(x-\alpha_m)^{l_m},\quad j=1,\ldots,t_i
\end{equation*}
and
\begin{equation*}
\omega_{\infty}=\prod_{m\in S_{\infty}}(x-\alpha_m)^{l_m}.
\end{equation*}
Moreover, we have $deg~h \leq (n-1)/\max\{t_i-2,1\}\leq n-1.$

(ii) Conversely for given data $P\in V_i(k), (k_1,\ldots, k_{t_i}),$ $S_{\infty}, S_{\beta_1},\ldots, 
S_{\beta_{t_i}},$ $(l_1,\ldots, l_n )$ as described in (i) one defines by the same equations 
rational functions $f, g, h$ with $f$ having at most $n$ zeros and poles altogether for which 
$f(x) = g(h(x))$ holds.

(iii) The integer $J$ and equations defining the varieties $V_i$ are effectively
computable only in terms of $n.$
\end{thmA}

Pethő and Tengely \cite{PT-CRF} provided some computational experiments that they obtained by using a MAGMA \cite{MAGMA} implementation of the algorithm of Fuchs and Peth\H{o} \cite{FPcrf}.

If the zeros and poles of a composite rational function form an arithmetic progression, then we have the following result.
\begin{thm}
Let $f,g,h$ be rational functions as in Theorem A. Assume that the zeros and poles of $f$ form an arithmetic progression, that is
$$
\alpha_i=\alpha_0+T_id
$$
for some $\alpha_0,d\in k$ and $T_i\in\{0,1,\ldots,n-1\}.$
If $k_1 + k_2 + \ldots + k_{t}\neq 0,$ then either the difference $d$ satisfies an equation of the form $$d^N=M$$
for some $N\in\mathbb{Z},M\in\mathbb{Q}$
or $(l_1,\ldots,l_n)\in\{0,1,\ldots,n-1\}^n$ satisfies a system of linear equations
$$
\sum_{r\in S_{\beta_i}}l_r=\sum_{s\in S_{\beta_j}}l_s,\quad i,j\in\{1,\ldots,t\},i\neq j.
$$
If $k_1 + k_2 + \ldots + k_{t}=0$ and $1\leq j_1<j_2<j_3\leq t,$ then 
$$d^{\sum_{m_1\in S_{\beta_{j_1}}}l_{m_1}},d^{\sum_{m_2\in S_{\beta_{j_2}}}l_{m_2}},d^{\sum_{m_3\in S_{\beta_{j_3}}}l_{m_3}}$$
satisfy a system of linear equations and
$\beta_{j_1},\beta_{j_2},\beta_{j_3}$
satisfy a system of linear equations.
\end{thm}
We will apply the above theorem to determine composite rational functions having 4 zeros and poles. We prove the following statement.
\begin{prop}\label{allitas}
Let $k$ be an algebraically closed field of characteristic zero. If $f, g, h \in k(x)$ with $f(x)=g(h(x))$ and with $\deg g, \deg h \geq 2, g$ not
of the shape $(\lambda(x))^m , m \in \N, \lambda \in PGL_2(k),$ and $f$ has 4 zeros and poles altogether forming an arithmetic progression, then $f$ is equivalent to the following rational function
$$
(x-\alpha_0)^{k_1}(x-\alpha_0-d)^{k_2}(x-\alpha_0-2d)^{k_2}(x-\alpha_0-3d)^{k_1},
$$
for some $\alpha_0,d\in k$ and $k_1,k_2\in\mathbb{Z}, k_1+k_2\neq 0.$
\end{prop}

In this paper we correct results obtained in \cite{PT-CRF}, where the computations related to the case $k_1 + k_2 + \ldots + k_{t}\neq 0, S_{\infty}=\emptyset$
are missing. 
The following theorem is the corrected version of Theorem 1 from \cite{PT-CRF}, where part (c) was missing. We define equivalence of rational functions.
Two rational functions $f_1(x)=\prod_{u=1}^n(x-\alpha_u^{(1)})^{f_u^{(1)}}$ and $f_2(x)=\prod_{u=1}^n(x-\alpha_u^{(2)})^{f_u^{(2)}}$
are equivalent if there exist $a_{u,v}\in\mathbb{Q},u\in\{1,2,\ldots,n\},v\in\{1,2,\ldots,n+1\}$ such that 
$$
\alpha_u^{(1)}=a_{u,1}\alpha_1^{(2)}+a_{u,2}\alpha_2^{(2)}+\ldots+a_{u,n}\alpha_n^{(2)}+a_{u,n+1},
$$
for all $u\in\{1,2,\ldots,n\}.$
\begin{thm}
Let $k$ be an algebraically closed field of characteristic zero. If $f, g, h \in k(x)$ with $f(x)=g(h(x))$ and with $\deg g, \deg h \geq 2, g$ not
of the shape $(\lambda(x))^m , m \in \N, \lambda \in PGL_2(k),$ and $f$ has 3 zeros and poles altogether, then $f$ is equivalent to one of the following rational functions
\begin{itemize}
\item[(a)] $\frac{(x-\alpha_1)^{k_1}(x+1/4-\alpha_1)^{2k_2}}{(x-1/4-\alpha_1)^{2k_1+2k_2}}$ for some $\alpha_1\in k$ and $k_1,k_2\in\mathbb{Z}, k_1+k_2\neq 0,$
\item[(b)] $\frac{(x-\alpha_1)^{2k_1}(x+\alpha_1-2\alpha_2)^{2k_2}}{(x-\alpha_2)^{2k_1+2k_2}}$ for some $\alpha_1,\alpha_2\in k$ and $k_1,k_2\in\mathbb{Z}, k_1+k_2\neq 0,$
\item[(c)] $\left(x-\frac{\alpha_1+\alpha_2}{2}\right)^{2k_1}(x-\alpha_1)^{k_2}(x-\alpha_2)^{k_2}$ for some $\alpha_1,\alpha_2\in k$ and $k_1,k_2\in\mathbb{Z}, k_1+k_2\neq 0.$
\end{itemize}
\end{thm}
\begin{rem}
The MAGMA procedure \texttt{CompRatFunc.m} can be downloaded from 
\texttt{http://shrek.unideb.hu/$\sim$tengely/CompRatFunc.m}.
All systems in cases of $n\in\{3,4,5\}$ can be downloaded from\\ \texttt{http://shrek.unideb.hu/$\sim$tengely/CFunc345.tar.gz}.
\end{rem}
\begin{rem}
It is interesting to note that in the above formulas the zeros and poles form an arithmetic progression
\begin{eqnarray*}
&& \mbox{(a):}\quad \alpha_1-\frac{1}{4},\alpha_1,\alpha_1+\frac{1}{4}\quad\mbox{ difference: }\frac{1}{4},\\
&& \mbox{(b):}\quad \alpha_1,\alpha_2,-\alpha_1+2\alpha_2\quad\mbox{ difference: }\alpha_2-\alpha_1,\\
&& \mbox{(c):}\quad \alpha_1,\frac{\alpha_1+\alpha_2}{2},\alpha_2\quad\mbox{ difference: }\frac{\alpha_2-\alpha_1}{2}.
\end{eqnarray*}
\end{rem}
\section{Auxiliary results}
We repeat some parts of the proof of Theorem A from \cite{FPcrf} that will be used here later on.
Without loss of generality we may assume that $f$ and $g$ are monic. Let
$$
f(x)=\prod_{i=1}^{n}(x-\alpha_i)^{f_i}
$$
with pairwise distinct $\alpha_i\in k$ and $f_i\in\Z$ for $i=1,\ldots,n.$ 
Similarly, let
$$
g(x)=\prod_{j=1}^{t}(x-\beta_j)^{k_j}
$$
with pairwise distinct $\beta_j\in k$ and $k_j\in\Z$ for $j=1,\ldots, t$ and $t\in\N.$ 
Hence we have
$$
\prod_{i=1}^{n}(x-\alpha_i)^{f_i}=f(x)=g(h(x))=\prod_{j=1}^{t}(h(x)-\beta_j)^{k_j}.
$$
We shall write $h(x) = p(x)/q(x)$ with $p, q \in k[x], p, q$ coprime.
Fuchs and Peth\H{o} \cite{FPcrf} showed that if $k_1+k_2+\ldots+k_{t}\neq 0,$ then there is a subset $S_{\infty}$ of the set $\{1,\ldots,n\}$
for which
$$
q(x)=\prod_{m\in S_{\infty}}(x-\alpha_m)^{l_m}
$$
and there is a partition of the set $\{1,\ldots,n\}\setminus S_{\infty}$ in $t$ disjoint non empty subsets 
$S_{\beta_1},\ldots, S_{\beta_{t}}$ such that
\begin{equation}\label{non-empty}
h(x)=\beta_j+\frac{1}{q(x)}\prod_{m\in S_{\beta_j}}(x-\alpha_m)^{l_m},
\end{equation}
where $l_m\in\N$ satisfies $l_m k_j = f_m$ for $m \in S_{\beta_j},$ and this holds true for every
$j=1,\ldots,t.$  
We get at least two different representations of $h,$ since we assumed that $g$ is not of the special shape
$(\lambda(x))^m.$ Therefore we get at least one equation of the form
\begin{equation}\label{EQ1}
\beta_i+\frac{1}{q(x)}\prod_{r\in S_{\beta_i}}(x-\alpha_r)^{l_r}=\beta_j+\frac{1}{q(x)}\prod_{s\in S_{\beta_j}}(x-\alpha_s)^{l_s}.
\end{equation}
If $k_1+k_2+\ldots+k_{t}=0,$ then we have 
$$
(p(x)-\beta_jq(x))^{k_j}=\prod_{m\in S_{\beta_j}}(x-\alpha_m)^{f_m}.
$$
Now we have that $t\geq 3,$ otherwise $g$ is in the special form we excluded. Siegel's identity provides the equations
in this case. That is if $1\leq j_1<j_2<j_3\leq t,$ then we have 
\begin{equation}\label{EQ2}
v_{j_1,j_2,j_3}+v_{j_3,j_1,j_2}+v_{j_2,j_3,j_1}=0,
\end{equation}
where
$$
v_{j_1,j_2,j_3}=(\beta_{j_1}-\beta_{j_2})\prod_{m\in S_{\beta_{j_3}}}(x-\alpha_m)^{l_m}.
$$
\section{Proofs of Theorem 1 and Theorem 2}
\begin{proof}[Proof of Theorem 1] 
If $k_1+k_2+\ldots+k_{t}\neq 0$ and there exist $r_1\in S_{\beta_i},s_1\in S_{\beta_j}$ for some $i\neq j$ such that $l_{r_1}\neq 0$ and $l_{s_1}\neq 0,$ then it follows from \eqref{EQ1} that
\begin{eqnarray}
\beta_i-\beta_j&=&\frac{\prod_{s\in S_{\beta_j}}(\alpha_{r_1}-\alpha_s)^{l_s}}{\prod_{m\in S_{\infty}}(\alpha_{r_1}-\alpha_m)^{l_m}},\\
\beta_i-\beta_j&=&-\frac{\prod_{r\in S_{\beta_i}}(\alpha_{s_1}-\alpha_r)^{l_r}}{\prod_{m\in S_{\infty}}(\alpha_{s_1}-\alpha_m)^{l_m}}
\end{eqnarray}
for any appropriate $\alpha_{r_1}\in S_{\beta_i}$ and $\alpha_{s_1}\in S_{\beta_j}.$
Hence we obtain that
$$
C_1(i,j,r_1,s_1)=d^{\sum_{r\in S_{\beta_i}}l_r-\sum_{s\in S_{\beta_j}}l_s},
$$
where $C_1(i,j,r_1,s_1)\in\mathbb{Q}.$ 
If there exist $S_{\beta_i}$ and $S_{\beta_j}$ for which 
$\sum_{r\in S_{\beta_i}}l_r-\sum_{s\in S_{\beta_j}}l_s\neq 0,$ then the possible values of $d$ satisfy equations of the form $x^N=M.$
Otherwise we get that 
$$
\sum_{r\in S_{\beta_i}}l_r=\sum_{s\in S_{\beta_j}}l_s,\quad i,j\in\{1,\ldots,t\},i\neq j.
$$

Let us consider the special case when $l_r=0$ for all $r\in S_{\beta_i}.$ If $l_s=0$ for all $s\in S_{\beta_j},$ then we get that 
$$
h(x)=\beta_i+\frac{1}{q(x)}=\beta_j+\frac{1}{q(x)}.
$$
Hence $\beta_i=\beta_j$ for some $i\neq j,$ a contradiction. Thus we may assume that there exists $s_1\in S_{\beta_j}$ for which $l_{s_1}\neq 0.$ In a similar way as in the above case it follows that 
\begin{eqnarray}
\beta_i-\beta_j&=&\frac{\prod_{s\in S_{\beta_j}}(\alpha_{r_1}-\alpha_s)^{l_s}}{\prod_{m\in S_{\infty}}(\alpha_{r_1}-\alpha_m)^{l_m}}-\frac{1}{\prod_{m\in S_{\infty}}(\alpha_{r_1}-\alpha_m)^{l_m}},\\
\beta_i-\beta_j&=&-\frac{1}{\prod_{m\in S_{\infty}}(\alpha_{s_1}-\alpha_m)^{l_m}}.
\end{eqnarray}
Therefore 
$$
d^{\sum_{s\in S_{\beta_j}}l_s}=C_2(i,j,r_1,s_1),
$$
where $C_2(i,j,r_1,s_1)\in\mathbb{Q}.$ Since $s_1>0$ we have that $\sum_{s\in S_{\beta_j}}l_s\neq 0,$ that is $d$ satisfies an appropriate polynomial equation.

If $k_1+k_2+\ldots+k_t=0,$ then there are at least 3 partitions and for any appropriate $r_1\in S_{\beta_{j_1}},r_2\in S_{\beta_{j_2}},r_3\in S_{\beta_{j_3}}$ (that is $l_{r_i}\neq 0,i=1,2,3$) equation \eqref{EQ2} implies that
{\small
\begin{eqnarray*}
(\beta_{j_3}-\beta_{j_1})\prod_{m_2\in S_{\beta_{j_2}}}(\alpha_{r_3}-\alpha_{m_2})^{l_{m_2}}+(\beta_{j_2}-\beta_{j_3})\prod_{m_1\in S_{\beta_{j_1}}}(\alpha_{r_3}-\alpha_{m_1})^{l_{m_1}}&=&0\\
(\beta_{j_1}-\beta_{j_2})\prod_{m_3\in S_{\beta_{j_3}}}(\alpha_{r_2}-\alpha_{m_3})^{l_{m_3}}+(\beta_{j_2}-\beta_{j_3})\prod_{m_1\in S_{\beta_{j_1}}}(\alpha_{r_2}-\alpha_{m_1})^{l_{m_1}}&=&0\\
(\beta_{j_1}-\beta_{j_2})\prod_{m_3\in S_{\beta_{j_3}}}(\alpha_{r_1}-\alpha_{m_3})^{l_{m_3}}+(\beta_{j_3}-\beta_{j_1})\prod_{m_2\in S_{\beta_{j_2}}}(\alpha_{r_1}-\alpha_{m_2})^{l_{m_2}}&=&0,
\end{eqnarray*}
}
that is a system of linear equations in $d_1,d_2,d_3,$
where $d_i=d^{\sum_{m_i\in S_{\beta_{j_i}}}l_{m_i}},i\in\{1,2,3\}$ and the statement follows. In a very similar way we obtain a system of equations if $l_{r}=0$ for all $r\in S_{\beta_{j_3}},$ the last two equations are as before, while on the left-hand side of the first one there is an additional term $\beta_{j_1}-\beta_{j_2}.$
\end{proof}
\begin{proof}[Proof of Theorem 2]
In \cite{PT-CRF} all cases are given with $k_1 + k_2 + \ldots + k_{t}=0$ and also with $k_1 + k_2 + \ldots + k_{t}\neq 0, S_{\infty}\neq\emptyset.$
Therefore it remains to deal with those cases with $k_1 + k_2 + \ldots + k_{t}\neq 0, S_{\infty}=\emptyset.$ First let $t=2.$ There are 18 systems of equations. Among these systems there 
are two types. The first one has only a single equation, e.g. when $S_{\beta_1}=\{1,2\},S_{\beta_2}=\{3\},(l_1,l_2,l_3)=(1,0,1),$ this equation is as follows
$$
\alpha_1-\alpha_3-\beta_1+\beta_2=0.
$$
Hence
$$
h(x)=\beta_1+(x-\alpha_1)=\beta_2+(x-\alpha_3)
$$
is a linear function. A system from the second type is given by $S_{\beta_1}=\{1,2\},S_{\beta_2}=\{3\},(l_1,l_2,l_3)=(1,1,2)$ and equations as follows
\begin{eqnarray*}
\alpha_1+\alpha_2-2\alpha_3&=&0\\
(\alpha_2-\alpha_3)^2-\beta_1+\beta_2&=&0.
\end{eqnarray*}
That is we obtain that
\begin{eqnarray*}
h(x)&=&\beta_2+\left(x-\frac{\alpha_1+\alpha_2}{2}\right)^2,\\
g(x)&=&\left(x-\beta_2-\left(\frac{\alpha_2-\alpha_1}{2}\right)^2\right)^{k_1}(x-\beta_2)^{k_2},\\
f(x)&=&\left(x-\frac{\alpha_1+\alpha_2}{2}\right)^{2k_1}(x-\alpha_1)^{k_2}(x-\alpha_2)^{k_2}.
\end{eqnarray*}
It is a decomposition of type (c) in the theorem.
Let $t=3.$ There are 6 systems of equations, all of the same type, e.g. $S_{\beta_1}=\{1\},S_{\beta_2}=\{2\},S_{\beta_3}=\{3\},(l_1,l_2,l_3)=(1,1,1)$
and
\begin{eqnarray*}
\alpha_1-\alpha_3-\beta_1+\beta_3&=&0\\
\alpha_2-\alpha_3-\beta_2+\beta_3&=&0.
\end{eqnarray*}
Hence the degree of $h$ is 1, that yields a trivial decomposition.
\end{proof}

\section{Proof of Proposition \ref{allitas}}
\begin{proof}[Proof of Proposition \ref{allitas}]
In this section we apply Theorem 1 to determine composite rational functions having zeros and poles 
as consecutive elements of certain arithmetic progressions. We need to handle the following cases
\begin{eqnarray*}
(I)&:& n=4 \mbox{ and } t\in\{2,3,4\}, k_1 + k_2 + \ldots + k_{t}\neq 0, S_{\infty}=\emptyset,\\
(II)&:& n=4 \mbox{ and } t\in\{2,3\}, k_1 + k_2 + \ldots + k_{t}\neq 0, S_{\infty}\neq\emptyset,\\
(III)&:& n=4 \mbox{ and } t\in\{3,4\}, k_1 + k_2 + \ldots + k_{t}=0, S_{\infty}=\emptyset.\\
\end{eqnarray*}
In the proof we use the notation of Theorem 1, that is we write 
$$
\alpha_i=\alpha_0+T_id,
$$
where $\alpha_0,d\in k$ and $\{T_1,T_2,T_3,T_4\}=\{0,1,2,3\}.$

\underline{$(I): t=2, \{|S_{\beta_1}|,|S_{\beta_2}|\}=\{1,3\}.$}
We may assume that $S_{\beta_1}=\{1\},S_{\beta_2}=\{2,3,4\}.$
We obtain that
\begin{eqnarray*}
	h(x)&=&\beta_1+(x-\alpha_1)^{l_1},\\
	h(x)&=&\beta_2+(x-\alpha_2)^{l_2}(x-\alpha_3)^{l_3}(x-\alpha_4)^{l_4}.
\end{eqnarray*}
Substituting $x=\alpha_2,\alpha_3,\alpha_4$ yields (assuming $l_2l_3l_4\neq 0$)
$$
(\alpha_2-\alpha_1)^{l_1}=(\alpha_3-\alpha_1)^{l_1}=(\alpha_4-\alpha_1)^{l_1}.
$$
Since the zeros and poles form an arithmetic progression one gets that either $d=0$ or $l_1=0.$ In the former case the zeros and poles are not distinct, a contradiction. In the latter case the degree of $h$ is less than 2, a contradiction as well. If two out of $l_2,l_3,l_4$  are equal to zero, then it follows that $l_1=1,$ hence the degree of $h$ is 1, a contradiction. If exactly one out of $l_2,l_3,l_4$ is zero, then $l_1=2$ and the corresponding $f$ has only 3 zeros and poles. As an example we consider the case $l_4=0.$ We obtain that 
$$
\alpha_1=\frac{\alpha_2+\alpha_3}{2}\quad\mbox{and}\quad\beta_2=\beta_1+\left(\frac{\alpha_2-\alpha_3}{2}\right)^2.
$$
It follows that $f(x)=\left(x-\frac{\alpha_2+\alpha_3}{2}\right)^2f_1(x),$
where $\deg f_1=2.$

\underline{$(I): t=2, \{|S_{\beta_1}|,|S_{\beta_2}|\}=\{2\}.$} Here we may assume that $S_{\beta_1}=\{1,2\},S_{\beta_2}=\{3,4\}.$
We get that
\begin{eqnarray*}
h(x)&=&\beta_1+(x-\alpha_1)^{l_1}(x-\alpha_2)^{l_2},\\
h(x)&=&\beta_2+(x-\alpha_3)^{l_3}(x-\alpha_4)^{l_4}.
\end{eqnarray*}
It follows that (assuming that $0\notin\{l_1,l_2,l_3,l_4\}$)
$$
(\alpha_1-\alpha_3)^{l_3}(\alpha_1-\alpha_4)^{l_4}=(\alpha_2-\alpha_3)^{l_3}(\alpha_2-\alpha_4)^{l_4}
$$
and 
$$
(\alpha_3-\alpha_1)^{l_1}(\alpha_3-\alpha_2)^{l_2}=(\alpha_4-\alpha_1)^{l_1}(\alpha_4-\alpha_2)^{l_2}.
$$
Using the fact that the zeros and poles form an arithmetic progression it turns out that one has to deal with 80 cases.
\begin{itemize}
\item There are 8 cases with $(l_1,l_2,l_3,l_4)=(1,1,1,1).$ We obtain equivalent solutions, so we only consider one of these. Let $\alpha_1=\alpha_0, \alpha_2=\alpha_0+3d.$ It follows that $\beta_2=\beta_1-2d^2.$ That is we have
\begin{eqnarray*}
g(x)&=&(x-\beta_1)(x-\beta_1+2d^2),\\
h(x)&=&\beta_1+(x-\alpha_0)(x-\alpha_0-3d),\\
f(x)&=&(x-\alpha_0)(x-\alpha_0-d)(x-\alpha_0-2d)(x-\alpha_0-3d).
\end{eqnarray*} 
\item There are 16 equivalent cases with $(l_1,l_2,l_3,l_4)\in\{(1,1,2,2),(2,2,1,1)\}.$
One obtains that $d^2=\pm\frac{1}{2}$ and $\beta_2=\beta_1\pm 1.$ One example from this family is given by 
\begin{eqnarray*}
	g(x)&=&(x-\beta_1)(x-\beta_1-1),\\
	h(x)&=&\beta_1+(x-\alpha_0-\sqrt{2}/2)^2(x-\alpha_0-\sqrt{2})^2,\\
	f(x)&=&(x-\alpha_0)\left(x-\alpha_0-\frac{\sqrt{2}}{2}\right)^2(x-\alpha_0-\sqrt{2})^2\left(x-\alpha_0-\frac{3\sqrt{2}}{2}\right)f_2(x),
\end{eqnarray*}
where $f_2(x)$ is a quadratic polynomial such that $f$ has more than 4 zeros and poles. We remark that if we use the equations related to $\beta_2$ we have
\begin{eqnarray*}
	g(x)&=&(x-\beta_2)(x-\beta_2+1),\\
	h(x)&=&\beta_2+(x-\alpha_0)(x-\alpha_0-3\sqrt{2}),\\
	f(x)&=&(x-\alpha_0)\left(x-\alpha_0-\frac{\sqrt{2}}{2}\right)(x-\alpha_0-\sqrt{2})\left(x-\alpha_0-\frac{3\sqrt{2}}{2}\right),
\end{eqnarray*}
that is we obtain a "solution" covered by the family given by the case $(l_1,l_2,l_3,l_4)=(1,1,1,1).$
\item There are 8 equivalent cases with  $(l_1,l_2,l_3,l_4)=(2,2,2,2).$ All of these cases can be eliminated in the same way. From the equation 
\begin{equation}\label{eq_d}
(\alpha_1-\alpha_3)^{l_3}(\alpha_1-\alpha_4)^{l_4}=-(\alpha_3-\alpha_1)^{l_1}(\alpha_3-\alpha_2)^{l_2}
\end{equation}
it follows that 
$$
d^{l_1+l_2-l_3-l_4}=\frac{(T_1-T_3)^{l_3}(T_1-T_4)^{l_4}}{-(T_3-T_1)^{l_1}(T_3-T_2)^{l_2}},
$$
where $\{T_1,T_2,T_3,T_4\}=\{0,1,2,3\}.$ The left-hand side is $d^0=1$ and the right-hand side is -1, a contradiction.
\item There are 16 equivalent cases with $(l_1,l_2,l_3,l_4)\in\{(1,1,3,3),(3,3,1,1)\}.$ As an example we handle the one with $(l_1,l_2,l_3,l_4)=(3,3,1,1)$ and
\begin{eqnarray*}
\alpha_1&=&\alpha_0,\\
\alpha_2&=&\alpha_0+3d,\\
\alpha_3&=&\alpha_0+2d,\\
\alpha_4&=&\alpha_0+d.
\end{eqnarray*}
Equation \eqref{eq_d} implies that either $d=0$ or $d^4=\frac{1}{4}.$ If $d^2=\frac{1}{2},$ then we get
\begin{eqnarray*}
	g(x)&=&(x-\beta_1)(x-\beta_1+1),\\
	h(x)&=&\beta_1+(x-\alpha_0)^3(x-\alpha_0-3\sqrt{2}/2)^3,\\
	f(x)&=&(x-\alpha_0)^3\left(x-\alpha_0-\frac{\sqrt{2}}{2}\right)(x-\alpha_0-\sqrt{2})\left(x-\alpha_0-\frac{3\sqrt{2}}{2}\right)^3f_3(x),
\end{eqnarray*}
where $f_3(x)$ is a quartic polynomial resulting an $f$ having more than 4 zeros and poles. 
If $d^2=-\frac{1}{2},$ then we get
\begin{eqnarray*}
	g(x)&=&(x-\beta_1)(x-\beta_1-1),\\
	h(x)&=&\beta_1+(x-\alpha_0)^3(x-\alpha_0-3\sqrt{-2}/2)^3,\\
	f(x)&=&(x-\alpha_0)^3\left(x-\alpha_0-\frac{\sqrt{-2}}{2}\right)(x-\alpha_0-\sqrt{-2})\left(x-\alpha_0-\frac{3\sqrt{-2}}{2}\right)^3f_4(x),
\end{eqnarray*}
where $f_4$ is a quartic polynomial and we get a contradiction in the same way as before.
\item There are 16 equivalent cases with $(l_1,l_2,l_3,l_4)\in\{(2,2,3,3),(3,3,2,2)\}.$ We handle the case with $(l_1,l_2,l_3,l_4)=(3,3,2,2)$ and 
\begin{eqnarray*}
	\alpha_1&=&\alpha_0+3d,\\
	\alpha_2&=&\alpha_0,\\
	\alpha_3&=&\alpha_0+2d,\\
	\alpha_4&=&\alpha_0+d.
\end{eqnarray*}
It follows from equation \eqref{eq_d} that $d=0$ or $d^2=\frac{1}{2}.$ Also we have that $\beta_2=\beta_1-1.$ In a similar way as in the above cases we obtain a composite function $f$ having 4 zeros and poles forming an arithmetic progression, but an additional quartic factor appears, a contradiction.
\item There are 8 equivalent cases with $(l_1,l_2,l_3,l_4)=(3,3,3,3).$ Here we consider the case with 
\begin{eqnarray*}
	\alpha_1&=&\alpha_0,\\
	\alpha_2&=&\alpha_0+3d,\\
	\alpha_3&=&\alpha_0+d,\\
	\alpha_4&=&\alpha_0+2d.
\end{eqnarray*}
It follows that $\beta_2=\beta_1-8d^6.$ As in the previous cases $g(h(x))$ has 4 zeros and poles coming from an arithmetic progression, but there is an additional quartic factor yielding a contradiction.

If $0 \in\{l_1,l_2,l_3,l_4\},$ then we have three possibilities. Either $\{l_1,l_2\}=\{l_3,l_4\}=\{0,1\}$ or $\{l_1,l_2\}=\{1\},\{l_3,l_4\}=\{0,2\}$ or $\{l_1,l_2\}=\{0,2\},\{l_3,l_4\}=\{1\}.$ In the first case the degree of $h$ is 1, a contradiction. The last two cases can be handled in the same way, therefore we only deal with the case $\{l_1,l_2\}=\{1\},\{l_3,l_4\}=\{0,2\}.$ Without loss of generality we may assume that $l_3=2,l_4=0.$ It follows that $\alpha_1=2\alpha_3-\alpha_2$ and $\beta_2=\beta_1-(\alpha_2-\alpha_3)^2.$ Thus
\begin{eqnarray*}
h(x)&=&\beta_1+(x-2\alpha_3+\alpha_2)(x-\alpha_2),\\
g(x)&=&(x-\beta_1)(x-\beta_1+(\alpha_2-\alpha_3)^2),\\
f(x)&=&(x-\alpha_2)(x-\alpha_3)^2(x-2\alpha_3+\alpha_2).
\end{eqnarray*}
We conclude that $f(x)$ has only 3 zeros and poles, a contradiction.

\underline{$(I): t=3, |S_{\beta_1}|=|S_{\beta_2}|=1, |S_{\beta_3}|=2.$}
Here we may assume that $S_{\beta_1}=\{1\},S_{\beta_2}=\{2\},S_{\beta_3}=\{3,4\},$ that is one has
\begin{eqnarray*}
	h(x)&=&\beta_1+(x-\alpha_1)^{l_1},\\
	h(x)&=&\beta_2+(x-\alpha_2)^{l_2},\\
	h(x)&=&\beta_3+(x-\alpha_3)^{l_3}(x-\alpha_4)^{l_4},\\
\end{eqnarray*}
where $l_1,l_2\in\{2,3\}.$ Let us consider the case $l_3\neq 0, l_4\neq 0.$ Substitute $\alpha_3,\alpha_4$ into the above system of equations to get
\begin{eqnarray*}
	\beta_3&=&\beta_1+(\alpha_3-\alpha_1)^{l_1},\\
	\beta_3&=&\beta_2+(\alpha_3-\alpha_2)^{l_2},\\
	\beta_3&=&\beta_1+(\alpha_4-\alpha_1)^{l_1},\\
	\beta_3&=&\beta_2+(\alpha_4-\alpha_2)^{l_2}.
\end{eqnarray*}
These equations imply that $\alpha_i=\alpha_j$ for some $i\neq j,$ a contradiction.
Now assume that $l_4=0,$ hence $l_3=2$ or 3. We can reduce the system as follows
\begin{eqnarray*}
(\alpha_1-\alpha_2)^{l_2}+(\alpha_2-\alpha_1)^{l_1}&=&0,\\
(\alpha_1-\alpha_3)^{l_3}+(\alpha_3-\alpha_1)^{l_1}&=&0,\\
(\alpha_2-\alpha_3)^{l_3}+(\alpha_3-\alpha_2)^{l_2}&=&0,
\end{eqnarray*}
where $l_1,l_2,l_3\in\{2,3\}.$ We get a contradiction in all these cases.

\noindent\underline{$(I): t=4, S_{\beta_1}=\{1\},S_{\beta_2}=\{2\},S_{\beta_3}=\{3\},S_{\beta_4}=\{4\}.$}
We obtain the system of equations
\begin{eqnarray*}
	h(x)&=&\beta_1+(x-\alpha_1)^{l_1},\\
	h(x)&=&\beta_2+(x-\alpha_2)^{l_2},\\
	h(x)&=&\beta_3+(x-\alpha_3)^{l_3},\\
	h(x)&=&\beta_4+(x-\alpha_4)^{l_4},
\end{eqnarray*}
where $l_i\geq 2$ (since $\deg h\geq 2.$) Here we prove that this type of composite rational
function cannot exist. One has that for any different $i,j$
$$
(\alpha_i-\alpha_j)^{l_j-l_i}=(-1)^{l_i+1}.
$$
If $l_i=l_j=2,$ then we have a contradiction. Assume that $l_i=2.$ There exist $l_j=l_k=3.$
Hence $\alpha_i=\alpha_j-1$ and $\alpha_i=\alpha_k-1,$ a contradiction.
Let us deal with the case $(l_1,l_2,l_3,l_4)=(3,3,3,3).$ Substituting $\alpha_1+\alpha_2$ into
the system of equations yields $\beta_1=\beta_2+\alpha_1^3-\alpha_2^3.$ We also have that
$\beta_1=\beta_2+(\alpha_1-\alpha_2)^3.$ By combining these equations we get that
$$
-3\alpha_1\alpha_2(\alpha_1-\alpha_2)=0.
$$
In a similar way we obtain
$$
-3\alpha_3\alpha_4(\alpha_3-\alpha_4)=0.
$$
It follows that for some different $i,j$ one has $\alpha_i=\alpha_j,$ a contradiction.

\noindent\underline{$(II): t=2,|S_{\infty}|=2,|S_{\beta_1}|=|S_{\beta_2}|=1 .$} We may assume that $S_{\infty}=\{1,2\},S_{\beta_1}=\{3\},S_{\beta_2}=\{4\}.$
The system of equations in this case is as follows
\begin{eqnarray*}
	h(x)&=&\beta_1+\frac{(x-\alpha_3)^{l_3}}{(x-\alpha_1)^{l_1}(x-\alpha_2)^{l_2}},\\
	h(x)&=&\beta_2+\frac{(x-\alpha_4)^{l_4}}{(x-\alpha_1)^{l_1}(x-\alpha_2)^{l_2}}.
\end{eqnarray*}
If $l_3=l_4=0,$ then it follows that $\beta_1=\beta_2,$ a contradiction. Let us deal with the case $l_3=0,l_4\neq 0$ (in a similar way one can handle the case $l_3\neq 0, l_4=0$). There are only three systems to consider. If $(l_1,l_2,l_3,l_4)=(0,1,0,1)$ or $(1,0,0,1),$ then $\beta_1-1=\beta_2$ and the composite function $f$ has only 2 zeros and poles, a contradiction. If $(l_1,l_2,l_3,l_4)=(1,1,0,2),$ then $\beta_1-1=\beta_2$ and $\alpha_4=\alpha_2\pm 1, \alpha_1=\alpha_2\pm 2.$ In all these cases we obtain a composite function $f$ having only 3 zeros and poles, a contradiction. Let us consider the cases with $l_3\neq 0,l_4\neq 0.$ There are 18 systems to deal with. It turns out that $d$ satisfies the equation
$$
d^{l_4-l_3}=-\frac{(T_4-T_3)^{l_3}(T_3-T_1)^{l_1}(T_3-T_2)^{l_2}}{(T_4-T_1)^{l_1}(T_4-T_2)^{l_2}(T_3-T_4)^{l_4}},
$$
where $\alpha_i=\alpha_0+T_id$ for some $T_i\in\{0,1,2,3\}.$
If $(l_1,l_2,l_3,l_4)=(1,0,2,2),$ then $$(T_1,T_2,T_3,T_4)\in\{(1,3,0,2),(1,3,2,0),(2,0,1,3),(2,0,3,1)\}.$$ In all these cases we obtain a composite function $f$ having only 3 zeros and poles, a contradiction. As an example we compute $f$ when $(T_1,T_2,T_3,T_4)=(1,3,0,2).$ We get that $\beta_2=\beta_1+4d$ and 
\begin{eqnarray*}
h(x)&=&\beta_1+\frac{(x-\alpha_0)^2}{(x-\alpha_0-d)},\\
g(x)&=&(x-\beta_1)(x-\beta_1-4d),\\
f(x)&=&\frac{{\left(x-\alpha_0-2d\right)}^{2} {\left(x-\alpha_0\right)}^{2}}{{\left(x-\alpha_0-d\right)}^{2}}.
\end{eqnarray*}
We exclude the tuple $(l_1,l_2,l_3,l_4)=(0,1,2,2)$ following the same lines.
If $(l_1,l_2,l_3,l_4)=(1,1,1,2),$ then we also have that $d=\frac{1}{T_1+T_2-2T_4}$ and $d=\frac{T_2-T_3}{(T_2-T_4)^2},$ it is easy to check that such tuple $(T_1,T_2,T_3,T_4)$ does not exist. In a very similar way if $(l_1,l_2,l_3,l_4)=(1,1,2,1)$ we obtain that
$$
d=\frac{1}{T_1+T_2-2T_3}=\frac{T_2-T_4}{(T_2-T_3)^2}
$$
and such tuple $(T_1,T_2,T_3,T_4)$ does not exist. If $(l_1,l_2,l_3,l_4)=(2,1,2,3),$ then 
\begin{eqnarray*}
\frac{(T_3-T_4)^3}{(T_3-T_1)^2(T_3-T_2)}&=&1,\\
-\frac{(T_4-T_3)^2}{(T_4-T_1)^2(T_4-T_2)}&=&\frac{4}{27(T_3-T_4)}.
\end{eqnarray*}
There is no solution in $T_i\in\{0,1,2,3\}, T_i\neq T_j, i\neq j.$ We obtain a very similar system of equations in case of $(l_1,l_2,l_3,l_4)=(1,2,3,2),(1,2,2,3),(2,1,3,2).$ If $(l_1,l_2,l_3,l_4)=(1,1,3,3),$ then we get
\begin{eqnarray*}
T_1+T_2&=&T_3+T_4,\\
(T_4-T_1)(T_4-T_2)&=&(T_3-T_1)(T_3-T_2),\\
27(T_2-T_4)^4(T_4-T_1)^2&=&9(T_4-T_3)^3(T_2-T_4)^2(T_4-T_1)-(T_4-T_3)^6.
\end{eqnarray*}	
The above system has no solution in $(T_1,T_2,T_3,T_4).$ If $(l_1,l_2,l_3,l_4)=(1,2,3,1),$ then
\begin{eqnarray*}
	T_1-4T_3+3T_4&=&0,\\
	2T_2+T_3-3T_4&=&0,\\
	(T_4-T_3)^3&=&(T_4-T_1)(T_4-T_2)^2.
\end{eqnarray*}
The system has no solution. The same argument works in case of $(l_1,l_2,l_3,l_4)=(1,2,1,3),(2,1,1,3),(2,1,3,1).$ If $(l_1,l_2,l_3,l_4)=(0,2,2,1),$ then we have
\begin{eqnarray*}
\alpha_2&=&\alpha_4+\frac{1}{4},\\
\alpha_3&=&\alpha_4-\frac{1}{4},
\end{eqnarray*}
hence
\begin{eqnarray*}
	h(x)&=&\beta_1+\frac{(x-\alpha_4+\frac{1}{4})^2}{(x-\alpha_4-\frac{1}{4})^2},\\
	g(x)&=&(x-\beta_1)(x-\beta_1-1),\\
	f(x)&=&\frac{(x-\alpha_4)(x-\alpha_4+\frac{1}{4})^2}{(x-\alpha_4-\frac{1}{4})^4}.
\end{eqnarray*}
That is $f$ has only 3 zeros and poles, a contradiction. We handle in the same way the tuples $(l_1,l_2,l_3,l_4)=(2,0,2,1),(2,0,1,2),(0,2,1,2).$ If $(l_1,l_2,l_3,l_4)=(0,0,1,1),$ then $\deg h(x)=1,$ a contradiction.
\end{itemize}

\noindent\underline{$(II): t=3,|S_{\infty}|=|S_{\beta_1}|=|S_{\beta_2}|=|S_{\beta_3}|=1 .$} We may assume that $S_{\infty}=\{1\},S_{\beta_1}=\{2\},S_{\beta_2}=\{3\},S_{\beta_3}=\{4\}.$
In this case $h(x)$ can be written as follows
\begin{eqnarray*}
	h(x)&=&\beta_1+\frac{(x-\alpha_2)^{l_2}}{(x-\alpha_1)^{l_1}},\\
	h(x)&=&\beta_2+\frac{(x-\alpha_3)^{l_3}}{(x-\alpha_1)^{l_1}},\\
	h(x)&=&\beta_3+\frac{(x-\alpha_4)^{l_4}}{(x-\alpha_1)^{l_1}}.
\end{eqnarray*}
The only possible exponent tuple $(l_1,l_2,l_3,l_4)$ is $(0,1,1,1).$ Thus $\deg h(x)=1,$ a contradiction.

\noindent\underline{$(III): t=3,|S_{\beta_1}|=2,|S_{\beta_2}|=|S_{\beta_3}|=1 .$} We may assume that $S_{\beta_1}=\{1,2\},S_{\beta_2}=\{3\},S_{\beta_3}=\{4\}.$ The only exponent tuple for which $\deg h(x)>1$ is given by $(l_1,l_2,l_3,l_4)$ is $(1,1,2,2).$ We obtain the following system of equations if $d\neq 0:$
\begin{eqnarray*}
(\beta_3-\beta_1)(T_4-T_3)^2+(\beta_2-\beta_3)(T_4-T_1)(T_4-T_2)&=&0\\
(\beta_1-\beta_2)(T_3-T_4)^2+(\beta_2-\beta_3)(T_3-T_1)(T_3-T_2)&=&0\\
(\beta_1-\beta_2)(T_1-T_4)^2+(\beta_3-\beta_1)(T_1-T_3)^2&=&0\\
(\beta_1-\beta_2)(T_2-T_4)^2+(\beta_3-\beta_1)(T_2-T_3)^2&=&0,
\end{eqnarray*}
where $\{T_1,T_2,T_3,T_4\}=\{0,1,2,3\}.$ Solving the above system of equations for all possible tuples $(T_1,T_2,T_3,T_4)$ one gets that $\beta_i=\beta_j$ for some $i\neq j,$ a contradiction.

\noindent\underline{$(III): t=3,|S_{\beta_1}|=|S_{\beta_2}|=|S_{\beta_3}|=|S_{\beta_4}|=1 .$} We may assume that $S_{\beta_1}=\{1\},S_{\beta_2}=\{2\},S_{\beta_3}=\{3\},S_{\beta_4}=\{4\}.$ The only possible exponent tuple is $(l_1,l_2,l_3,l_4)=(1,1,1,1).$ Thus the corresponding $h(x)$ has degree 1, a contradiction. As an example we consider the case 
\begin{eqnarray*}
\alpha_1&=&\alpha_0+d,\\
\alpha_2&=&\alpha_0,\\
\alpha_3&=&\alpha_0+3d,\\
\alpha_4&=&\alpha_0+2d.\\
\end{eqnarray*}
We use equation \eqref{EQ2} here with $(j_1,j_2,j_3)=(1,2,3)$ and $(j_1,j_2,j_3)=(1,2,4).$ If $d\neq 0,$ then we have
\begin{eqnarray*}
	\beta_3&=&3\beta_1-2\beta_2,\\
	\beta_4&=&2\beta_1-\beta_2.
\end{eqnarray*}
Let $k_1,k_2,k_3,k_4\in\mathbb{Z}$ such that $k_1+k_2+k_3+k_4=0.$ Theorem A implies that
\begin{eqnarray*}
	g(x)&=&(x-\beta_1)^{k_1}(x-\beta_2)^{k_2}(x-3\beta_1+2\beta_2)^{k_3}(x-2\beta_1+\beta_2)^{k_4},\\
	h(x)&=&\frac{1}{d}(\beta_1(x-\alpha_0)-\beta_2(x-\alpha_0-d)),\\
	f(x)&=&(x-\alpha_0-d)^{k_1}(x-\alpha_0)^{k_2}(x-\alpha_0-3d)^{k_3}(x-\alpha_0-2d)^{k_4}.
\end{eqnarray*}
\end{proof}
\section{Cases with $n=4$}
In this section we provide some details of the computation corresponding to cases with $n=4,t\in\{2,3,4\},k_1 + k_2 + \ldots + k_{t}\neq 0, S_{\infty}=\emptyset.$
These are the cases which are not mentioned in Section 5 in \cite{PT-CRF}.

\noindent\underline{The case $n=4, t=2$ and $S_{\infty}=\emptyset.$}
There are 134 systems to deal with. We treat only a few representative examples.

If $S_{\beta_1}=\{1,2\},S_{\beta_2}=\{3,4\}$ and $(l_1,l_2,l_3,l_4)=(2,1,2,1),$ then we have
\begin{eqnarray*}
\alpha_1 + 1/2\alpha_2 - \alpha_3 - 1/2\alpha_4&=&0\\
\alpha_2 - 4/3\alpha_3 + 1/3\alpha_4&=&0\\
\alpha_2\alpha_3^2 - 2\alpha_2\alpha_3\alpha_4 + \alpha_2\alpha_4^2 - \alpha_3^2\alpha_4 + 2\alpha_3\alpha_4^2 - \alpha_4^3 - 9\beta_1 + 9\beta_2&=&0\\
\alpha_2 - 4/3\alpha_3 + 1/3\alpha_4&=&0\\
\alpha_3^3 - 3\alpha_3^2\alpha_4 + 3\alpha_3\alpha_4^2 - \alpha_4^3 - 27/4\beta_1 + 27/4\beta_2&=&0.
\end{eqnarray*}
The corresponding rational functions are as follows
\begin{eqnarray*}
f(x)&=&(x-\alpha_1)^{2k_1}(x-\alpha_2)^{k_1}(x-\frac{1}{3}\alpha_1-\frac{2}{3}\alpha_2)^{2k_2}(x-\frac{4}{3}\alpha_1+\frac{1}{3}\alpha_2)^{k_2},\\ 
g(x)&=&(x-\beta_1)^{k_1}(x-\beta_1-\frac{4}{27}(\alpha_1-\alpha_2)^3)^{k_2}\\
h(x)&=&\beta_1+(x-\alpha_1)^2(x-\alpha_2),
\end{eqnarray*}
where $k_1+k_2\neq 0.$ We note that the zeros and poles of $f$ do not form an arithmetic progression for all values of the parameters as the choice $\alpha_1=0,\alpha_2=3$ shows.

If $S_{\beta_1}=\{1,2\},S_{\beta_2}=\{3,4\}$ and $(l_1,l_2,l_3,l_4)=(1,1,0,2),$ then we get the system of equations
\begin{eqnarray*}
\alpha_1+\alpha_2-2\alpha_4&=&0\\
(\alpha_2-\alpha_4)^2-\beta_1+\beta_2&=&0.
\end{eqnarray*}
It yields a decomposable rational function $f$ having only 3 zeros and poles altogether.

If $S_{\beta_1}=\{1,2\},S_{\beta_2}=\{3,4\}$ and $(l_1,l_2,l_3,l_4)=(1,1,1,1),$ then we obtain
\begin{eqnarray*}
\alpha_1 + \alpha_2 - \alpha_3 - \alpha_4&=&0\\
\alpha_2^2 - \alpha_2\alpha_3 - \alpha_2\alpha_4 + \alpha_3\alpha_4 - \beta_1 + \beta_2&=&0. 
\end{eqnarray*}
It yields the following solution
\begin{eqnarray*}
f(x)&=&(x+\alpha_2-\alpha_3-\alpha_4)^{k_1}(x-\alpha_2)^{k_1}(x-\alpha_3)^{k_2}(x-\alpha_4)^{k_2},\\ 
g(x)&=&(x-\beta_1)^{k_1}(x-\beta_1+\alpha_2^2 - \alpha_2\alpha_3 - \alpha_2\alpha_4 + \alpha_3\alpha_4)^{k_2}\\
h(x)&=&\beta_1+(x-\alpha_3-\alpha_4+\alpha_2)(x-\alpha_2),
\end{eqnarray*}
where $k_1+k_2\neq 0.$

If $S_{\beta_1}=\{1,2,3\},S_{\beta_2}=\{4\}$ and $(l_1,l_2,l_3,l_4)=(1,1,1,3),$ then we have
\begin{eqnarray*}
\alpha_1 + \alpha_2 + \alpha_3 - 3\alpha_4&=&0\\
\alpha_2^2 + \alpha_2\alpha_3 - 3\alpha_2\alpha_4 + \alpha_3^2 - 3\alpha_3\alpha_4 + 3\alpha_4^2&=&0\\
\alpha_3^3 - 3\alpha_3^2\alpha_4 + 3\alpha_3\alpha_4^2 - \alpha_4^3 - \beta_1 + \beta_2&=&0.
\end{eqnarray*}
We obtain the following rational functions
\begin{eqnarray*}
f(x)&=&(x-\alpha_1)^{k_1}(x-\alpha_2)^{k_1}(x-\alpha_3)^{k_1}(x-\alpha_4)^{3k_2},\\ 
g(x)&=&(x-\beta_2-(\alpha_3-\alpha_4)^3)^{k_1}(x-\beta_2)^{k_2}\\
h(x)&=&\beta_2+(x-\alpha_4)^3,
\end{eqnarray*}
where $k_1+k_2\neq 0$ and
\begin{eqnarray*}
\alpha_1&=&\frac{1}{2} \, \alpha_{4} {\left(-i \, \sqrt{3} +3\right)} - \frac{1}{2} \, \alpha_{3} {\left(-i \,\sqrt{3} + 1\right)}\\
\alpha_2&=&\frac{1}{2} \, \alpha_{4} {\left(i \, \sqrt{3} +3\right)} + \frac{1}{2} \, \alpha_{3} {\left(-i \,\sqrt{3} - 1\right)}.
\end{eqnarray*}

\noindent\underline{The case $n=4, t=3$ and $S_{\infty}=\emptyset.$}
There are 48 systems to handle in this case. We consider one of these. Let $S_{\beta_1}=\{1\},S_{\beta_2}=\{2,3\},S_{\beta_3}=\{4\}$ and $(l_1,l_2,l_3,l_4)=(1,1,0,1).$
We obtain the system of equations
\begin{eqnarray*}
\alpha_1 - \alpha_4 - \beta_1 + \beta_3&=&0\\
\alpha_2 - \alpha_4 - \beta_2 + \beta_3&=&0. 
\end{eqnarray*}
It follows that $h$ is a linear function, which only provides trivial decomposition. In the remaining cases we have the same conclusion.

\noindent\underline{The case $n=4, t=4$ and $S_{\infty}=\emptyset.$}
Here we get 24 systems to consider. In all cases we have that
$$
\{S_{\beta_1},S_{\beta_2},S_{\beta_3},S_{\beta_4}\}=\{\{1\},\{2\},\{3\},\{4\}\}
$$
and $(l_1,l_2,l_3,l_4)=(1,1,1,1).$ Therefore $h$ is linear, a contradiction.
\bibliography{all}
\end{document}